\documentclass[11pt]
{amsart}
\usepackage{amssymb,amsthm,amsmath,latexsym}

\usepackage{float}
\usepackage{url}
\usepackage{systeme}
\usepackage{wrapfig}
\usepackage[margin=0.9in]{geometry}
\usepackage{tikz}

\usetikzlibrary{shapes,backgrounds,arrows.meta}
\usepackage{tikz-3dplot}

\usepackage{caption}

\usepackage{quoting}
\def\cal{\mathcal}

\newtheorem{theorem}{Theorem}

\newtheorem*{proposition}{Proposition}

\newtheorem*{remark*}{Remark}

\makeatletter
\renewcommand*\env@matrix[1][*\c@MaxMatrixCols c]{% 
  \hskip -\arraycolsep
  \let\@ifnextchar\new@ifnextchar
  \array{#1}}
\makeatother

\usepackage{hyperref}
\usepackage[justification=centering]{caption}

\begin{document}

\title
{Helly's Theorem--A Very Early Introduction
}
\markright{ \, Helly's Theorem--Early Introduction }

\author{Eric L.~Grinberg}
\date{} % This suppresses the date

\newenvironment{acknowledgement}%       New acknowledgement environment
    {\large\bfseries Acknowledgement%
    \par\medskip\normalfont\normalsize}%
    {}%

\begin{abstract}
We propose an interpretation of, and approach to, Helly's theorem that can be included quite early in the undergraduate curriculum. At the same time, the approach connects with contemporary models of data privacy and with sampling methods used in epidemiology. The presentation is intended to be accessible to teachers and their students. \end{abstract}


\maketitle

%\vspace*{2cm}
\begin{center}
    \begin{minipage}{0.7\textwidth}
        \centering\small\textit{Dedicated to Boris Rubin on the occasion of his 80th birthday}
    \end{minipage}
\end{center}

%\author{ Eric Grinberg }
%\address{UMass Boston }
%\email{eric.grinberg@umb.edu }

\section{Introduction.}

Helly's theorem engenders excitement when introduced to students, yet many math majors meet it late in their studies, if at all. It can engender just as much excitement among majors in engineering, economics, and other fields. A linear algebra class often has representation from these student groups and is a viable forum for an early introduction to Helly. So, in the first days of this semester, we gave students a glimpse of Helly's intersection theorem in the context of linear algebra. Returning to the office, we found a new issue of the \emph{Bulletin of the American Mathematical Society} featuring, what else, but the celebrated Eduard Helly and his intersection theorems \cite{barany22, hauser22}. 
Coincidence?  Perhaps. In any event, paralleling the idea of \emph{Early Transcendentals} in calculus, we thought we’d share what one might call an \emph{Early Helly} \hyperlink{fig:cartoon}[H] approach in linear algebra, as well as in basic set theory.

\section{Overdetermined systems: sampling for consistency.} 

\begin{wrapfigure}[9]{r}{0.1\textwidth}
\vspace{-23pt} 
%\hspace{-10pt} 
\includegraphics[width= 0.1\textwidth]{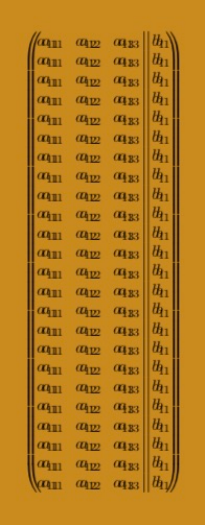}
\end{wrapfigure}

\smallskip
Imagine a large, overdetermined system of equations, with many more equations than unknowns, say one hundred equations in three variables. Is such a system consistent (read: does it have a solution?) Determining consistency directly may require substantial computation. One can consider taking ``samples'' of small subsystems that are only slightly overdetermined. Such sampling ideas also appear in modern contexts such as epidemiological testing and privacy-preserving data analysis. For instance, take a sample of four equations out of the one hundred, and test the smaller system for consistency. One can hope that if the entire system is inconsistent then a sample subsystem will already reveal the inconsistency. But before embarking on such a statistical strategy, one might want some kind of guarantee against strong false positives.

\section{Conventions and notations.}
A linear equation is said to be \emph{consistent} if it has a solution; it is said to be \emph{degenerate}
if all coefficients of the unknowns are zero. Thus $0x+0y+0z=1$ is an inconsistent degenerate equation and  $0x+0y+0z=0$
is a consistent degenerate equation.  Each of the equations in \eqref{tetrahedron-system} below is consistent and nondegenerate.

\section{A tetrahedral example.}

\begin{equation} \label{tetrahedron-system}
\systeme{x+0y+0z=0,0x+y+0z=0,0x+0y+z=0,
x+y+z=1}  \qquad .
\end{equation}

In the linear system \eqref{tetrahedron-system} each of the four equations is consistent. The solution sets are the $yz$-plane, the $xz$-plane, and the $xy$-plane, respectively.  The fourth (last) equation has a solution set that is a slanted plane. Choosing any three of the four equations, we obtain a consistent subsystem. For instance, if we choose the first three equations then the resulting subsystem has the solution $x=0,y=0,z=0$, i.e., the origin.  If we choose the subsystem comprising the first two and the last equation then the solution set is  $x=0,y=0,z=1$. Proceeding in this way, the reader can verify the following.

\begin{proposition}
Any $3$-equation subsystem of the tetrahedral  linear system \eqref{tetrahedron-system} is consistent, but the system as a whole is inconsistent.
\end{proposition}

Reflecting on this result, we see that consistent subsystems do not guarantee consistent overall systems. Still, a variation on this theme yields positive results. Let's  ``up'' the size of the sampled subsystems. 

\section[Four equations suffice]{Four equations suffice.\protect\footnote{This title is in homage to the Four Color Theorem, and its motto \emph{four colors suffice} \cite{Wilson13}.}}

\begin{theorem}[Helly for Planes in $\mathbb R^3$] \, \\
Let $T$ be a nondegenerate  system of $N$ linear equations in $3$ unknowns, with $N \ge 4$. Assume that any subsystem of $T$ comprised of $4$ equations is consistent. Then the full system is consistent.
\end{theorem}

\begin{proof}By nondegeneracy, each equation in the system $T$ has a solution set that is a plane. If all $N$ planes are the same, the solution set of $T$ is that plane, and the system is consistent. 

Otherwise, there are two equations, $E_1,E_2$, whose solution sets are distinct planes $P_1,P_2$. These planes must meet since we can append two additional equations to $E_1,E_2$ to form a consistent $4$-equation subsystem, and solutions to that subsystem belong to both planes. Thus $P_1 \cap P_2$ is a line $\ell$, since two distinct intersecting planes in $3$-space meet in a line.

If all the rest of the solution planes meet $P_1 \cap P_2$  in $\ell$ then $\ell$ is part of the solution set of $T$ and hence $T$ is consistent.  If some solution plane $P_3$ meets $P_1 \cap P_2$ in less than the full line, that solution plane meets $\ell$ at just one point $Q$. Each additional solution plane beyond $P_1,P_2,P_3$  must also meet $\ell$ at $Q$, else the resulting $4$-equation subsystem would be inconsistent. Hence the point $Q$ belongs to all the solution planes and so $T$ is consistent.
\end{proof}

\subsection{Follow-up exercises.} For entry-level students: work out the case of systems of equations with two unknowns, so that any three-equation subsystem is consistent;  use the geometry of lines in the plane. For upper level students: work out the case of equations with $k$ unknowns so that any  $(k+1)$ equation subsystem is consistent; use the concept of vector space \emph{dimension}.

\section{Recalling Venn Diagrams}

We recall (in more ways than one) the tool of \emph{Venn diagrams}, often found in introductory proof classes. For example, here is a Venn diagram illustrating three sets so that any two meet, but not all three meet:

\def\firstcircle{(-0.6,-0.6) circle (1.5cm)}
\def\secondcircle{(60:2cm) circle (1.5cm)}
\def\thirdcircle{(-17:2.2cm) circle (1.4cm)}
\smallskip \begin{tikzpicture}
    \begin{scope}[shift={(3cm,-5cm)}, fill opacity=0.5]
        \fill[blue] \firstcircle;
        \fill[green] \secondcircle;
        \fill[red] \thirdcircle;
        \draw \firstcircle node[below] {$A$};
        \draw \secondcircle node [above] {$B$};
        \draw \thirdcircle node [below] {$C$};
    \end{scope}
\end{tikzpicture}

For configurations with more than three sets,  the standard Venn diagrams suffer limitations; perhaps, for this reason, they have been recalled. See, e.g., \cite[Exercise 4, p. 44]{velleman}. Helly's theorem serves to ``explain" this limitation. \\

Below we will say that two sets $A$ and $B$ \emph{meet} to indicate that their intersection, $A \cap B$, is nonempty.

\begin{theorem}[Minimalist Helly] Let  $\mathcal  C$  be  a collection of $N$ disks  in the plane, with $N \ge 3$. Assume that any \emph{three} disks in $\mathcal  C$  meet. Then all the disks of $\mathcal  C$  meet.
\end{theorem}

\begin{wrapfigure}[3]{r}{0.09\textwidth}
  \vspace{-22pt}
%\centering
 %\hspace{-20pt}    
 %
 % https://tex.stackexchange.com/questions/498769/tikz-define-3d-coordinate-axes-for-the-space-pyramid-tetrahedron
 %
\tdplotsetmaincoords{70}{100} 
\scalebox{0.25}{
\begin{tikzpicture}[line join=round,tdplot_main_coords,declare function={a=5;}] 
\begin{scope}[canvas is xy plane at z=0,transform shape]
 \path foreach \X [count=\Y] in {A,B,C}
 {(\Y*120:{a/(2*cos(30))}) coordinate(\X)};
\end{scope}
\path (0,0,{a*cos(30)}) coordinate (D);
\draw[line width=3.5pt]  foreach \X/\Y [remember=\X as \Z (initially D)] in {A/B,B/C,C/D,D/A}
 {(\X) -- (\Z) -- (\Y)};
\end{tikzpicture} }
\end{wrapfigure}
Thus, for example, Venn diagrams (planar, disk-based) cannot convey the intersection properties of the four sides of a tetrahedron. (Any three tetrahedral faces have a common meeting point; no point is common to all four.)
\\

Before proceeding with the proof, we consider the proper and improper ways in which disks can meet. If two disks in $\mathcal C$ meet \emph{properly}, their boundary circles intersect in a pair of points and the  intersection set of the disks is a region called a \emph{lens}. It is bounded by two arcs. See (a) below.

%\scalebox{1.3}
{
\[
\scalebox{1.6}{
\begin{tikzpicture}
  \draw[fill=blue!70,rotate=90]
  ([shift={(-40:1cm)}]0,0) arc (-40:40:1cm);
   \draw[fill=blue!70,rotate=-90]
  ([shift={(-40:1cm)}]-1.535,0) arc (-40:40:1cm);
 \node[right] (a) at (0.0,0.3){(a)};
\end{tikzpicture}
\quad \begin{tikzpicture}
  \draw[fill=blue!70,rotate=90]
  ([shift={(-40:1cm)}]0,0) arc (-40:40:1cm);
   \draw[fill=blue!70,rotate=-90]
  ([shift={(-40:1cm)}]-1.535,0) arc (-40:40:1cm);
\draw [<->,thick,rotate=20]  (0.5,0.45)  arc (0:28.5:01.cm);
 \node[right] (b) at (0.0,0.3){(b)};
\end{tikzpicture}
\quad
\begin{tikzpicture}
%\draw [lightgray] (0,0) grid (5,5);
  \draw[fill=blue!70,rotate=90]
  ([shift={(-40:1cm)}]0,0) arc (-40:40:1cm);
   \draw[fill=blue!70,rotate=-90]
  ([shift={(-40:1cm)}]-1.535,0) arc (-40:40:1cm);
%
%\hspace{-25pt} 
\draw (0.5,0.68)  [<->,thick,rotate=60] arc (0:70:0.75cm);
 \node[right] (a) at (0.0,0.3){(c)};
\end{tikzpicture}}
\]
}
A third disk may intersect the first two disks in less than this lens. Then the three disks intersect in a region which is a smaller lens bounded by two arcs (c), or a region bounded by three arcs (b). Of course, some disks can intersect ``improperly", and meet at just one point (\emph{osculate}), or one disk may be contained in another, as illustrated below:
 \\

\def\firstcircle{(-0.53,-0.6) circle (1.2cm)}
\def\secondcircle{(-20:2cm) circle (1.2cm)}
\begin{tikzpicture}
    \begin{scope}[shift={(2cm,-1cm)}, fill opacity=0.5]
        \fill[blue] \firstcircle;
        \fill[green] \secondcircle;
        \draw \firstcircle node[below] {$A$};
        \draw \secondcircle node [above] {$B$};
    \end{scope}
\end{tikzpicture}
\qquad \qquad
\begin{tikzpicture}
\def\firstcircle{(1.30,-0.6) circle (1.3cm)}
\def\secondcircle{(-20:2cm) circle (0.6cm)}
    \begin{scope}[shift={(2cm,-1cm)}, fill opacity=0.5]
        \fill[blue] \firstcircle;
        \fill[green] \secondcircle;
        \draw \firstcircle node[left] {$A$};
        \draw \secondcircle node [] {$B$};
    \end{scope}
\end{tikzpicture} . \\

In all situations, the intersection of a finite collection of disks is a region bounded by circular arcs, which may degenerate to a single point, or revert to a full circle.

\begin{proof} 

We induct on $N$, the number of disks in $\cal C$. The starter case $N=3$ is given by the hypothesis. We assume the result for all configurations  with $N$ disks, and consider a collection $\cal C$ with $N+1$ disks.

We may assume that $N$ of the disks intersect in a region bounded by circular arcs; call it $G$. Take the $N+1$st disk and call it $T$. If $T$ meets $G$, we are done. If $T$ does not meet $G$, consider the closest point-pair  between the disk $T$ and the arc-polygonal region $G$. (Such a pair exists because $G$ is bounded by arcs.) Draw a directed line segment, from $T$ to $G$, connecting the closest point pair. There are two cases: the closest point on $G$ is in the interior of a boundary arc, or the closest point on $G$ is at a corner $P$, where two arcs meet.

\begin{figure}[H]
\centering
\begin{tabular}{|c|c|}
 \hline &  \\
\scalebox{1.1}{
\begin{tikzpicture}
[thick,
    set/.style = {circle,
        minimum size = 3cm,
        } ]
 
% Intersection
\begin{scope}
   \clip (0,0) circle(1.7cm);
   \clip (1.8,0) circle(1.5cm);
   \clip (0.9,1.5) circle(1.5cm);
   \fill[blue!55](0.9,1.50) circle(1.5cm);
\end{scope}
 
 \node[] (G) at (1.0,0.5){G};
 
 % Panel (i)
 \begin{scope}
    \clip (-3.2,-3.2) rectangle (3.2,3.2);
    \draw [thick,black!60!black] (-5.138,2.435) circle(5.0cm);
    \node[] (T) at (-2.0,1.8){T};
 \end{scope}
 
 \draw [green,thick, -{Stealth[scale=1.2]}] (-0.37,0.89) -- (0.44,0.627);
 \draw[thick, green, {Stealth[scale=1.2]}-{Stealth[scale=1.2]}] (0.402 - 0.526, 0.640 - 1.62) -- (0.402 + 0.526, 0.640 + 1.62);
 \draw [dashed] (1.78,0.01) circle(1.5cm);
\end{tikzpicture}  }
&
 \scalebox{1.1}{
\begin{tikzpicture}
[thick,
    set/.style = {circle,
        minimum size = 3cm,
        } ]
 
% Intersection
\begin{scope}
   \clip (0,0) circle(1.7cm);
   \clip (1.8,0) circle(1.5cm);
   \clip (0.9,1.5) circle(1.5cm);
   \fill[blue!55](0.9,1.50) circle(1.5cm);
\end{scope}

  \draw [dashed] (1.78,0.01) circle(1.5cm);
  \draw [dashed] (0.92,1.48) circle(1.5cm);
  \node[] (G) at (1.0,0.5){G};
 
 % Panel (ii)
 \begin{scope}
    \clip (-3.2,-3.2) rectangle (3.2,3.2);
    \draw [thick,black!60!black] (-3.637,-3.582) circle(5.0cm);
    \node[] at (-1.8, -1.8){T};
 \end{scope}

 % Precision Clipping to Tangent Rays
 \begin{scope}
    % The upper ray vector is (0.102, 1.314). Multiplying to reach outside the figure.
    % The lower ray vector is (1.205, -0.545).
    \clip (0.305, 0.125) -- (0.305 + 0.102*3, 0.125 + 1.314*3) -- (4, 4) -- (4, -2) -- (1.51, -0.42) -- cycle;
    
    % Radial lines: Semithick purple!70
    \foreach \ang in {-24, -14, ..., 96}
       \draw[purple!70, semithick] (0.305, 0.125) -- +(\ang:2.2cm);
       
    % Concentric arcs: Uniform spacing 0.4
    \foreach \r in {0.4, 0.8, 1.2, 1.6, 2.0}
       \draw[purple!70, semithick] (0.305, 0.125) circle (\r cm);
 \end{scope}
 
 \draw[thick,green, -{Stealth[scale=1.2]}] (-0.031,-0.175) -- (0.305,0.125);
 \draw[thick, green, {Stealth[scale=1.2]}-{Stealth[scale=1.2]}] (0.305 - 1.146, 0.125 + 1.283) -- (0.305 + 1.146, 0.125 - 1.283);

 % Purple Tangent Rays
 \draw[thick, purple, -{Stealth[scale=1.2]Stealth[scale=1.2]}] (0.305, 0.125) -- (0.407, 1.439);
 \draw[thick, purple, -{Stealth[scale=1.2]Stealth[scale=1.2]}] (0.305, 0.125) -- (1.51, -0.42);

\end{tikzpicture}   }
\\
\hline
(i) Arc interior is closest  & (ii) Corner is closest     \\
\hline
\end{tabular}
\end{figure}

If $T$ is closest to a point in the interior of an arc of $G$, as in \emph{(i)} above, the tangent line to $G$ at the closest point is perpendicular to the connecting line segment, and this tangent line separates $G$ from $T$, so $T$ and the disk bounded by the closest arc  do not meet. The hypothesis states that any three of the disks meet, so certainly any two of the disks meet, and we have reached a contradiction.

In the other case, $T$ is closest to a point $P$ on $G$ where two boundary arcs of $G$ meet, as in \emph{(ii)} above. The two circles corresponding to the two arcs that meet at $P$ have tangent \emph{rays} (double-headed in (ii)) that stay on the same side as $G$ relative to the perpendicular to the directed segment from $T$ to $G$. These two rays form a region (an angle), illustrated by a radial grid. The perpendicular to the segment going from $T$ to $P$ separates the disk $T$ from the radial grid, and thus also separates the two circles from $T$. Hence $T$ and the two respective disks do not meet, and we have reached a contradiction.
\end{proof}

\begin{remark*} Here we have focused on disks, in contrast with the general Helly theorem that allows for convex sets. The narrowed context has the advantage of avoiding the need for additional results, such as Radon's or Carath\'{e}odory's theorems, and thereby fits more easily and less distractedly into a first linear algebra course. The proof is indirect. It would be interesting to consider a more direct proof.
\end{remark*}

\begin{figure}[H]
\hypertarget{fig:cartoon}{} % Create a target for the link
\caption*{Figure H}
\centering
\includegraphics[width=0.6 \textwidth]{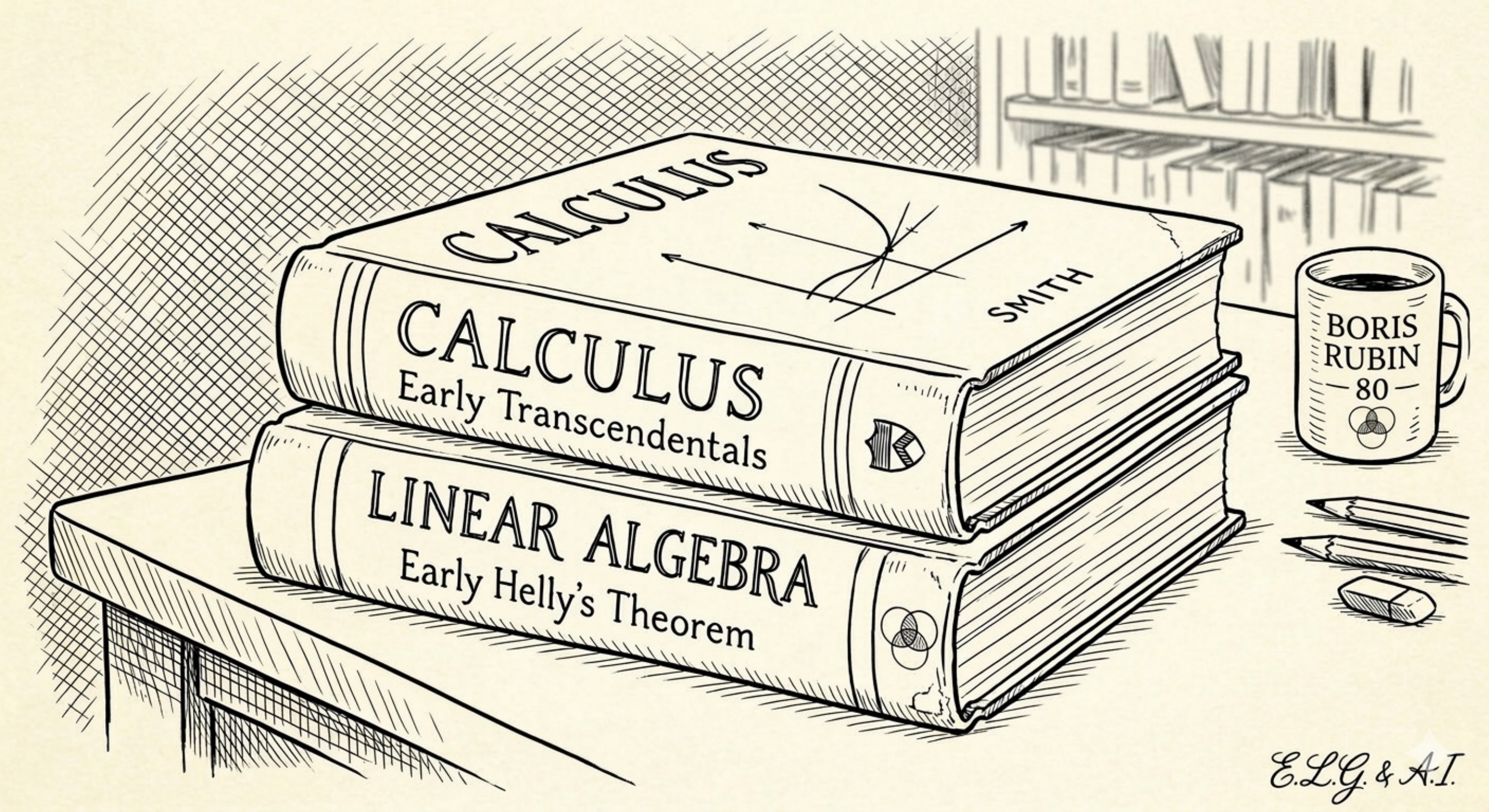}
\end{figure}

\noindent
Department of Mathematics, University of Massachusetts Boston,   USA \\
\qquad \href{mailto:eric.grinberg@umb.edu}{eric.grinberg@umb.edu}

\end{document}